\documentclass[11pt]{article}
\usepackage{amssymb,amsmath,amsfonts,amsthm}

\usepackage{epsfig}

\numberwithin{equation}{section}

\newcommand{\R}{{\mathbb R}}

\newcommand{\intr}{\int\limits_{\R^N}}

\theoremstyle{plain}

\newtheorem{theorem}{Theorem}[section]
\newtheorem{lemma}{Lemma}[section]

\theoremstyle{definition}

\begin{document}

\title{\vskip-0.3in Existence results for singular $p$-biharmonic problem with Hardy potential and critical Hardy-Sobolev exponent}

\author{Gurpreet Singh\footnote{School of Mathematics and Statistics, Technological University Dublin; {\tt gurpreet.bajwa2506@gmail.com}}}

\maketitle

\begin{abstract}
In this article, we consider the singular $p-$biharmonic problem involving Hardy potential and citical Hardy-Sobolev exponent. We study the existence of ground state solutions and least energy sign-changing solutions of the following problem
\begin{equation*}
\Delta_{p}^{2} u -\lambda_{1}  \frac{|u|^{p-2}u}{|x|^{2p}}= \frac{|u|^{p_{*}(\alpha)-2}}{|x|^{\alpha}}u+\lambda_{2}\Big(|x|^{-\beta}*|u|^{q}\Big)|u|^{q-2}u \quad\mbox{ in }\R^{N},
\end{equation*}
where $p>2$, $2<q< p_{*}(\alpha)$, $\lambda_{1}>0$, $\lambda_{2} \in \R$, $\alpha, \beta \in (0,N)$, $p_{*}(\alpha)=\frac{p(N-\alpha)}{N-2p}$ and $N\geq 5$. Firstly, we study existence of ground state solutions by using the minimization method on the associated Nehari manifold. Then, we investigate the least energy sign-changing solutions by considering the Nehari nodal set.
\end{abstract}




\section{Introduction}\label{sec1}
We study the existence of ground state solutions and least energy sign-changing solutions for the following problem
\begin{equation}\label{fr}
\Delta_{p}^{2} u -\lambda_{1}  \frac{|u|^{p-2}}{|x|^{2p}}= \frac{|u|^{p_{*}(\alpha)-2}}{|x|^{\alpha}}u+\lambda_{2}\Big(|x|^{-\beta}*|u|^{q}\Big)|u|^{q-2}u \quad\mbox{ in }\R^{N},
\end{equation}
where $p>2$, $2<q< p_{*}(\alpha)$, $\lambda_{1}>0$, $\lambda_{2} \in \R$, $\alpha, \beta \in (0,N)$ and $N\geq 5$. Here, $p_{*}{(\alpha)}= \frac{p(N-\alpha)}{N-2p}$ is the critical Sobolev exponent and $\Delta^{2}_{p} u = \Delta(|\Delta u|^{p-2}\Delta u)$ is the $p$-biharmonic operator and $|x|^{-\alpha} $ is the Riesz potential of order $\alpha \in (0, N)$.

Problems involving $p-$biharmonic operator, Hardy potential and singular non-linearities are of great importance as they appear in many applications, such as non-Newtonian fluids, viscous fluids, quantum mechanics, flame propagation, travelling waves in suspension bridges and many more (see \cite{BV2016}, \cite{CLR2006}, \cite{LM1990}, \cite{R2000}). These type of problems have received a considerable attention recently (see \cite{ADG2021}, \cite{CGK2023}, \cite{HL2014})

In \cite{DGR2023}, A. Drissi, A. Ghanmi and D. D. Repous considered the problem 
\begin{equation}\label{dgr}
\Delta_{p}^{2} u -\lambda  \frac{|u|^{p-2}}{|x|^{2p}}+\Delta_p u= \frac{|u|^{p_{*}(\alpha)-2}}{|x|^{\alpha}}u+\lambda_1 f(x)h(u) \quad\mbox{ in }\R^{N},
\end{equation}
where $0\leq \alpha< 2p,$ $1< p < \frac{N}{2},$ $\lambda, \lambda_1 > 0$ and $p_{*}(\alpha) = \frac{p(N-\alpha)}{N-2p}$. They studied the existence and multiplicity of solutions to \eqref{dgr}.

Wang \cite{W2020} had studied the following problem
\begin{equation}\label{wa}
\left\{
\begin{aligned}
\Delta_{p}^{2} u &= h(x,u)+ \lambda \frac{|u|^{r-2}}{|x|^{\alpha}}u &&\quad\mbox{ in }\Omega,\\
u&= \Delta u = 0 &&\quad\mbox{ on } \partial{\Omega},
\end{aligned}
\right.
\end{equation}
where $\Omega$ is a bounded domain. Wang has used the Mountain pass theorem to establish the existence results and Fountain theorem to find the existence of multiple solutions to \eqref{wa}.
When $p=2,$ $\lambda_2=0$, \eqref{fr} gives us the following biharmonic Choquard equation
\begin{equation}\label{squ}
\Delta^{2} u= \frac{|u|^{p_{*}(\alpha)-2}}{|x|^{\alpha}}u \quad\mbox{ in }\R^{N}.
\end{equation}

The Choquard equation has also received a lot of attention (see \cite{AFY2016, AGSY2017, GT2016, MV2017, MS2017}) and has appeared in many different contexts. For instance, In 1954 the following Choquard or nonlinear Schr\"odinger-Newton equation
\begin{equation}\label{ce}
-\Delta u+ u= (|x|^{-2}*u^2)u \quad\mbox{ in }\R^{N},
\end{equation}
was first studied by Pekar\cite{P1954} for $N=3$. Later in 1996, the equation \eqref{ce} was used by Penrose as a model in self-gravitating matter(see \cite{P1996}, \cite{P1998}). The following stationary Choquard equation
$$
-\Delta u+ V(x)u= (|x|^{-\alpha}*|u|^b)|u|^{b-2}u \quad\mbox{ in }\R^{N},
$$
arises in quantum theory and in the theory of Bose-Einstein condensation. 

Biharmonic equations have been studied widely and has many physical applications such as phase field models of multi-phase systems, micro electro-mechanical system, in thin film theory, nonlinear surface diffusion on solids, interface dynamics, biophysics, continuum mechanics, differential geometry, flow in Hele–Shaw cells, the deformation of a nonlinear elastic beam [18], and the Willmore equation(see \cite{GGS2010}). As the maximum principle cannot be applied to the biharmonic operator, this makes problems involving this operator even more interesting from a mathematical point of view(see \cite{BG2011, JC2014, PS2012, PS2014, YT2013, ZTZ2014}). In \cite{CD2018}, Cao and Dai had considered the following biharmonic equation with Hatree type nonlinearity
\begin{equation*}
\Delta^{2} u= \Big(|x|^{-8}*|u|^2\Big)|u|^{b}, \quad\mbox{ for all } x\in \R^d,
\end{equation*}
where $0< b\leq 1$ and $d\geq 9$. They used the methods of moving plane and able to prove that the non-negative classical solutions are radially symmetric. In the subcritical case $0<b<1$, they also able to get results on the non-existence of non-trivial non-negative classical solutions.

In \cite{MP1998}, Micheletti and Pistoiain has studied the following problem
\begin{equation*}
\left\{
\begin{aligned}
\Delta^{2} u+c\Delta u&= f(x, u) &&\quad\mbox{ in } \Omega,\\
u&= \Delta u= 0 &&\quad\mbox{ on } \partial{\Omega},
\end{aligned}
\right.
\end{equation*}
where $\Omega$ is smooth bounded domain in $\R^N$. They used the mountain pass theorem and obtained the multiple non-trivial solutions. Zhao and Xu in \cite{ZW2008} had studied the existence of infinitely many sign-changing solutions of the above problem by the use of critical point theorem.

\smallskip

In this paper, we study the existence of ground state solutions and the sign-changing solutions to equation \eqref{fr}. In the subsection below, we introduce some notations that will be used throughout this paper, present the variational framework, and state the main results.

\smallskip

\subsection{Notations and Variational Framework}
 
\begin{itemize}
\item $H_{0}^{p}(\R^N)= W_{0}^{2,p}(\R^N)$ is the Hilbert-Sobolev space and we will be denoting it by $E$ throughout this article.
\item The Hardy-Sobolev exponent is related to the following Rellich inequality
$$
\int_{\R^N} \frac{|u|^p}{|x|^{2p}} \;\; dx \leq \Big( \frac{p^2}{N(p-1)(N-2p)}\Big)^p \int_{\R^N} |\Delta u|^p \;\; dx, \quad\mbox{ for all } u \in E.
$$
As per Rellich inequality, $E$ can be endowed with the norm
$$
|| u ||_{E} = \Big( \int_{\R^N} |\Delta u|^p - \lambda_{1} \frac{|u|^p}{|x|^{2p}} \Big)^{\frac{1}{p}},
$$
provided 
$$0< \lambda_{1} < \Big( \frac{N(p-1)(N-2p)}{p^2}\Big)^p.$$
\item $L^{s}(\R^N)$ denotes the Lebesgue space in $\R^N$ of order $s\in [1 , \infty]$ with norm $||.||_{s}$.
\item $$||u||_{r} \leq S_{r}^{\frac{-1}{p}}||u||_{E},$$ for all $u\in E$ and $p\leq r \leq p_{*}$. Here,
$$S_{r} = \inf_{u\in E} \frac{\int_{\R^N} \Big(|\Delta u|^{p} - \lambda_1 \frac{|u|^p}{|x|^{2p}}\Big)}{\Big(\int_{\R^N} |x|^{-\alpha}|u|^r\Big)^{\frac{p}{r}}}.$$
\item $\hookrightarrow$ denotes the continuous embeddings.
\end{itemize}

\medskip

{\bf Note: }The embedding
$$
E=H_{0}^{p}(\R^N) \hookrightarrow L^{s}(\R^N),
$$
is continuous for all $p\leq s\leq p_{*}$ and is compact for all $p\leq s< p_{*}$.

\begin{itemize}
\item We will be using the following Hardy-Littlewood-Sobolev inequality
\begin{equation}\label{hli}
\Big| \int_{\R^N} \Big(|x|^{-\gamma}*u\Big)v \Big| \leq C\|u\|_r \|v\|_t,	
\end{equation}
for $\gamma \in (0, N)$, $u\in L^{r}(\R^N)$ and $v\in L^{t}(\R^N)$ such that
$$
\frac{1}{r}+ \frac{1}{t}+ \frac{\gamma}{N}= 2.
$$

\item Assume that $q$ satisfies 

\begin{equation}\label{p1}
\frac{p(2N-\beta)}{2N}< q< \frac{p(2N-\beta)}{2(N-2p)}.
\end{equation}

\item The equation \eqref{fr} has variational structure. We define the energy functional ${\mathcal J} : E \rightarrow \R$ by
\begin{equation}\label{fr1}
\begin{aligned}
{\mathcal J} (u)&=\frac{1}{p}\|u\|_{E}^{p}-\frac{1}{p_{*}(\alpha)}\int_{\R^{N}}|x|^{-\alpha}|u|^{p_{*}(\alpha)}-\frac{\lambda_{2}}{2q}\int_{\R^{N}}\Big(|x|^{-\beta}*|u|^{q}\Big)|u|^{q}.
\end{aligned}
\end{equation}
\end{itemize}
Using  Hardy-Littlewood-Sobolev inequality \eqref{hli} with \eqref{p1}, we get that the energy functional ${\mathcal J}$ is well defined  and ${\mathcal J}\in C^1(E)$. Also, a critical point of the energy functional  ${\mathcal J}$ is a solution of \eqref{fr} .

\subsection{Main Results.}  

Now, we present our main result on existence of ground state solution. We define the Nehari manifold associated with the energy functional $J$ by 
\begin{equation}\label{nm}
{\mathcal N}=\{u\in E\setminus\{0\}: \langle {\mathcal J}'(u),u\rangle=0\},
\end{equation}
and the ground state solutions will be obtained as minimizers of
$$
m=\inf_{u\in {\mathcal N}}{\mathcal J}(u).
$$

\begin{theorem}\label{nonexis}
Let $N\geq 5$, $\lambda_2> 0$, $p_{*}(\alpha)> 2q> p$ and $q$ satisfies \eqref{p1}. Then the equation \eqref{fr} has a ground state solution $u\in E$.
\end{theorem}

\smallskip

Next, we investigate the existence of least energy sign changing solutions for the equation \eqref{fr}. Now, we use the minimization method on the Nehari nodal set defined as 
$$
{\cal \overline {N}}= \Big\{u\in E: u^{\pm} \neq 0 \mbox{ and } \langle {\mathcal J}'(u),u^{\pm}\rangle  \mbox{ = } 0 \Big\},
$$
and solutions will be obtained as minimizers for 
$$
\overline{m}= \inf_{u\in {\cal \overline{N}}}{\mathcal J}(u).
$$
Here,
\begin{equation*}
\begin{aligned}
\langle {\mathcal J}'(u),u^{\pm} \rangle& = \|u^{\pm}\|_{E}^{p}-\frac{1}{p_{*}(\alpha)}\int_{\R^{N}}|x|^{-\alpha}(u^{\pm})^{p_{*}(\alpha)}-\frac{\lambda_2}{2q}\int_{\R^{N}}\Big(|x|^{-\beta}*(u^{\pm})^{q}\Big)(u^{\pm})^{q}-\frac{\lambda_2}{2q}\int_{\R^{N}}\Big(|x|^{-\beta}*(u^{\pm})^{q}\Big)(u^{\mp})^{q}.
\end{aligned}
\end{equation*}
\begin{theorem}\label{gstate}
Let $N\geq 5$, $\lambda_2 \in \R$, $p_{*}(\alpha)> 2q> p$ and $q$ satisfies \eqref{p1}. Then the equation \eqref{fr} has a least energy sign changing solution $u\in E$. 
\end{theorem}

\medskip

Now, in the Section $2$	, we will gather some preliminary results, which will be followed by Sections $3$ and $4$. In these sections, we will present the proofs of our main results.

\bigskip

\section{Preliminary results}

\begin{lemma}\label{cc}(\cite[Lemma 1.1]{L1984}, \cite[Lemma 2.3]{MV2013})
Let $u\in E$ and $e\in [p, p_{*}(\alpha)]$. There exists a constant $C_0>0$ such that
$$
\int_{\R^N}|u|^e \leq C_0||u||\Big(\sup_{y\in \R^N} \int_{B_1(y)}|u|^e \Big)^{1-\frac{2}{e}}.
$$

\end{lemma}

\begin{lemma}\label{bogachev}(\cite[Proposition 4.7.12]{B2007})
If a bounded sequence $(\mu_n)$ converges to $\mu$ almost everywhere in $L^s(\R^N)$ for some $s \in (1, \infty)$, then $\mu_n$ also converges weakly to $\mu$ in  $L^s(\R^N)$.
\end{lemma}

\begin{lemma}\label{blc}{\sf (\it Local Brezis-Lieb lemma)}
Let $(\mu_n)$ be a bounded sequence converging to $\mu$ almost everywhere in $L^s(\R^N)$ for some $s \in (1, \infty)$. Then, we have
$$
\lim_{n\to\infty}\int_{\R^N}\big| |\mu_n|^t-|\mu_n-\mu|^t-|\mu|^t\big|^{\frac{s}{t}}=0\,,
$$
and
$$
\lim_{n\to\infty}\int_{\R^N}\big| |\mu_n|^{t-1}\mu_n-|\mu_n-\mu|^{t-1}(\mu_n-\mu)-|\mu|^{t-1}\mu\big|^{\frac{s}{t}}=0,
$$
for all $1\leq t\leq s$,
\end{lemma}
\begin{proof}
Fix $\varepsilon>0$. Then, there exists a constant $C(\varepsilon)>0$ such that 
\begin{equation}\label{lb1}
\Big||c+d|^{t}-|c|^{t}\big|^{\frac{s}{t}}\leq \varepsilon|c|^{s}+C(\varepsilon)|d|^{s},
\end{equation}
for all $c$,$d\in \R$. Next, by \eqref{lb1}, we have
$$
\begin{aligned}
|f_{n, \varepsilon}|=& \Big( \Big||\mu_{n}|^{t}-|\mu_{n}-\mu|^{t}-|\mu^{t}|\Big|^{\frac{s}{t}}-\varepsilon|\mu_{n}-\mu|^{s}\Big)^{+}\\
&\leq (1+C(\varepsilon))|\mu|^{s}.
\end{aligned}
$$
Now, using Lebesgue Dominated Convergence theorem, one could get
\begin{equation}
\intr {f_{n, \varepsilon}} \rightarrow 0  \quad\mbox{ as } n\rightarrow \infty.
\end{equation}
Therefore, we have
$$
\Big||\mu_{n}|^{t}-|\mu_{n}-\mu|^{t}-|\mu|^{t}\Big|^{\frac{s}{t}}\leq f_{n, \varepsilon}+\varepsilon|\mu_{n}-\mu|^{s},
$$
and we obtain
$$
\limsup_{n\rightarrow \infty} {\intr \Big||\mu_{n}|^{t}-|\mu_{n}-\mu|^{t}-|\mu|^{t}\Big|^{\frac{s}{t}}}\leq C_{0}\varepsilon,
$$
where $C_{0}= \sup_{n}|\mu_{n}-\mu|_{q}^{q}< \infty$. We finish our proof by letting $\varepsilon \rightarrow 0$.
\end{proof}

\begin{lemma}\label{nlocbl}{\sf (\it Nonlocal Brezis-Lieb lemma}(\cite[Lemma 2.4]{MV2013})
Let $N\geq 5$, $q \in [1,\frac{2N}{2N-\beta})$ and assume that $(\mu_n)$ is a bounded sequence in $L^{\frac{2Nq}{2N-\beta}}(\R^N)$ such that $\mu_n \rightarrow \mu$ almost everywhere in $\R^N$. Then, we have
\begin{equation*}
\int_{\R^N}\Big(|x|^{-\beta}*{|\mu_n|^{q}}\Big){|\mu_n|^q}dx-\int_{\R^N}\Big(|x|^{-\beta}*{|\mu_n-\mu|^q}\Big){|\mu_n-\mu|^q}dx \rightarrow \int_{\R^N}\Big(|x|^{-\beta}*{|\mu|^q}\Big){|\mu|^q}dx.
\end{equation*}

\end{lemma}
\begin{proof}
Let $s=\frac{2N}{2N-\beta}$ in Lemma \ref{blc}, then one could get 
\begin{equation}\label{bh1}
|\mu_n-\mu|^{q}-|\mu_n|^{q}\to |\mu|^{q} \mbox{ strongly in } L^{\frac{2N}{2N-\beta}}(\R^N),
\end{equation}
as $n\rightarrow \infty$. Using Lemma \ref{bogachev}, we obtain 
\begin{equation}\label{bh2}
|\mu_n-\mu|^{q}\rightharpoonup 0 \mbox{ weakly in } L^{\frac{2N}{2N-\beta}}(\R^{N}).
\end{equation}
By the use of Hardy-Littlewood-Sobolev inequality \eqref{hli}, we have
\begin{equation}\label{b3}
|x|^{-\beta}*\Big(|\mu_n-\mu|^q-|\mu_n|^{q}\Big)\to |x|^{-\beta}*|\mu|^{q} \quad\mbox{ in } L^{\frac{2N}{\beta}}(\R^N).
\end{equation}
Also,
\begin{equation}\label{nb1}
\begin{aligned}
&\int_{\R^N}\Big(|x|^{-\beta}*|\mu_n|^{q}\Big)|\mu_n|^{q}dx-\int_{\R^N}\Big(|x|^{-\beta}*|\mu_n-\mu|^{q}\Big)|\mu_n-\mu|^{q} dx\\&=\intr \Big[|x|^{-\beta}*\Big(|\mu_n|^{q}-|\mu_n-\mu|^{q}\Big)\Big]\Big(|\mu_n|^{q}-|\mu_n-\mu|^{q}\Big)dx\\
&+2\intr \Big[|x|^{-\beta}*\Big(|\mu_n|^{q}-|\mu_n-\mu|^{q}\Big)\Big]|\mu_n-\mu|^{q} dx.
\end{aligned}
\end{equation}

Finally, passing to the limit in \eqref{nb1} and using \eqref{bh1}-\eqref{bh2}, the result holds.
	
\end{proof}

\begin{lemma}\label{anbl}
Suppose that $N\geq 5$, $\beta \in (0,N)$ and $q \in [1,\frac{2N}{2N-\beta})$. Let $(\mu_n)$ be a bounded sequence in $L^{\frac{2Nq}{2N-\beta}}(\R^N)$ such that $\mu_n \rightarrow \mu$ almost everywhere in $\R^N$. Then, for any $h\in L^{\frac{2Nq}{2N-\beta}}(\R^N)$ we have
$$
\int_{\R^N}\Big(|x|^{-\beta}*|\mu_n|^q\Big)|\mu_n|^{q-2}\mu_nh\; dx \rightarrow \int_{\R^N} \Big(|x|^{-\beta}*|\mu|^q\Big)|\mu|^{q-2}\mu h\; dx.
$$
\end{lemma}
\begin{proof} 
 Assume that $h=h^+-h^-$ and $\nu_n=\mu_n-\mu$. We prove the lemma for $h\geq 0$. Using Lemma \ref{blc} by taking $s=\frac{2N}{2N-\beta}$ together with $(z_n,z)=(\mu_n,\mu)$ and $(z_n,z)=(\mu_nh^{1/c}, \mu h^{1/c})$ respectively in order to obtain
$$
\left\{
\begin{aligned}
&|\mu_n|^{q}-|\nu_n|^{q}\to |\mu|^{q} \\
&|\mu_n|^{{q}-2}\mu_n h- |\nu_n|^{{q}-2}\nu_n h\to |\mu|^{{q}-2}\mu h
\end{aligned}
\right.
\quad\mbox{ strongly in }\; L^{\frac{2N}{2N-\beta}}(\R^N).
$$
Using the Hardy-Littlewood-Sobolev inequality, we obtain
\begin{equation}\label{est00}
\left\{
\begin{aligned}
&|x|^{-\beta}*\Big(|\mu_n|^{q}-|\nu_n|^{q}\Big)\to |x|^{-\beta}*|\mu|^{q} \\
&|x|^{-\beta}*\Big(|\mu_n|^{{q}-2}\mu_nh-|\nu_n|^{{q}-2}\nu_nh\Big)\to |x|^{-\beta}*\Big(|\mu|^{{q}-2}\mu h\Big)
\end{aligned}
\right.
\quad\mbox{ strongly in }\; L^{\frac{2N}{\beta}}(\R^N).
\end{equation}
Next, using Lemma \ref{bogachev} we get
\begin{equation}\label{est01}
\left\{
\begin{aligned}
&|\mu_n|^{{q}-2}\mu_n h\rightharpoonup |\mu|^{{q}-2}\mu h\\ 
& |\nu_n|^{q}\rightharpoonup 0\\
&|\nu_n|^{{q}-2}\nu_nh\rightharpoonup 0
\end{aligned}
\right.
\quad \mbox{ weakly in }\; L^{\frac{2N}{2N-\beta}}(\R^N)
\end{equation}
By equation \eqref{est00} and \eqref{est01}, we have
\begin{equation}\label{est02}
\begin{aligned}
&\intr \Big[|x|^{-\beta}*\Big(|\mu_n|^{q}-|\nu_n|^{q}\Big)\Big]\Big(|\mu_n|^{{q}-2}\mu_n h-|\nu_n|^{{q}-2}\nu_n h\Big)\to \int_{\R^N} \Big(|x|^{-\beta}*|\mu|^{q}\Big)|\mu|^{{q}-2}\mu h,\\
& \intr \Big[|x|^{-\beta}*\Big(|\mu_n|^{q}-|\nu_n|^{q}\Big)\Big]|\nu_n|^{{q}-2}\nu_nh\to 0,\\
&\intr \Big[|x|^{-\beta}*\Big(|\mu_n|^{{q}-2}\mu_nh-|\nu_n|^{{q}-2}\nu_nh\Big)\Big]|\nu_n|^{q}\to 0.
\end{aligned}
\end{equation}
Using the Hardy-Littlewood-Sobolev inequality together with H\"older's inequality, we find
\begin{equation}\label{est03}
\begin{aligned}
\left| \intr \Big(|x|^{-\beta}*|\nu_n|^{{q}}\Big)|\nu_n|^{{q}-2}\nu_nh \right|& \leq \|\nu_n\|^{q}_{\frac{2N{q}}{2N-\beta}}\||\nu_n|^{{q}-1}h\|_{\frac{2N}{2N-\beta}}\\
&\leq {q} \||\nu_n|^{{q}-1}h\|_{\frac{2N}{2N-\beta}}.
\end{aligned}
\end{equation}
Next, by Lemma \ref{bogachev} we have $\nu_n^{\frac{2N(q-1)}{2N-\beta}}\rightharpoonup 0$ weakly in $L^{\frac{q}{q-1}}(\R^N)$. Therefore, 
$$
\||\nu_n|^{q-1}h\|_{\frac{2N}{2N-\beta}}=\left(\intr |\nu_n|^{\frac{2N(q-1)}{2N-\beta}}|h|^{\frac{2N}{2N-\beta}}  \right)^{\frac{2N-\beta}{2N}}\to 0.
$$
Hence, using \eqref{est03} we have
\begin{equation}\label{est04}
\lim_{n\to \infty} \intr \Big(|x|^{-\beta}*|\nu_n|^{q}\Big)|\nu_n|^{q-2}\nu_nh=0.
\end{equation}
Also, one could notice that
\begin{equation}\label{bl1}
\begin{aligned}
\intr \Big(|x|^{-\beta}*|\mu_n|^{q}\Big)|\mu_n|^{{q}-2}\mu_nh=& \intr \Big[|x|^{-\beta}*\Big(|\mu_n|^{q}-|\nu_n|^{q}\Big)\Big]\Big(|\mu_n|^{{q}-2}\mu_nh-|\nu_n|^{{q}-2}\nu_nh\Big)\\
&+\intr \Big[|x|^{-\beta}*\Big(|\mu_n|^{q}-|\nu_n|^{q}\Big)\Big]|\nu_n|^{{q}-2}\nu_nh\\
&+\intr \Big[|x|^{-\beta}*\Big(|\mu_n|^{{q}-2}\mu_nh-|\nu_n|^{{q}-2}\nu_n h\Big)\Big]|\nu_n|^{q}\\
&+\intr \Big(|x|^{-\beta}*|\nu_n|^{q}\Big)|\nu_n|^{{q}-2}\nu_nh.
\end{aligned}
\end{equation}
 We get the desired result by passing to the limit in \eqref{bl1} together with \eqref{est02} and \eqref{est04}.
\end{proof}

\bigskip

We will now examine the ground state solutions to equation \eqref{fr} in the following section.

\medskip

\section{Proof of Theorem \ref{gstate}}
Here, in this section we will do the analysis of the Palais-Smale sequences for ${\mathcal J} \!\mid_{\mathcal N}$ and we will take the ideas from \cite{CM2016,CV2010} to prove that any Palais-Smale sequence of ${\mathcal J} \!\mid_{\mathcal N}$ is either converging strongly to its weak limit or differs from it by a finite number of sequences, which are the translated solutions of \eqref{squ}. 
Let us assume that $\lambda_2 > 0$. Then, for any $u,v\in E$  we have
\begin{equation*}
\begin{aligned}
\langle {\mathcal J}'(tu), tu \rangle &= t^{p}\|u\|_{E}^{p}- t^{ p_{*}(\alpha)}\int_{\R^{N}}|x|^{-\alpha}|u|^{p_{*}(\alpha)}-\lambda_{2} t^{2q}\int_{\R^{N}}\Big(|x|^{-\beta}*|u|^{q}\Big)|u|^{q},
\end{aligned}
\end{equation*}
where $t>0$.

The equation $\langle {\mathcal J}'(tu),tu \rangle= 0 $ has a unique positive solution $t=t(u)$ as $p_{*}(\alpha)>q> 1$. This unique positive solution is also known as the {\it projection of $u$} on ${\mathcal N}$. Next, we will be interested in the main properties of the Nehari manifold ${\mathcal N}$:

\begin{lemma}\label{nehari1}
\begin{itemize}
\item [(i)] Energy functional ${\mathcal J}$ is coercive, that is, ${\mathcal J} \geq c||u||_{E}^{p}$ for some constant $c>0$. 
\item [(ii)] ${\mathcal J} \!\mid_{\mathcal N}$ is bounded from below by a positive constant.
\end{itemize}
\end{lemma}
\begin{proof} 
\begin{itemize}
\item [(i)] We have,
$$
\begin{aligned}
{\mathcal J}(u)&= {\mathcal J}(u)-\frac{1}{2q}\langle {\mathcal J}'(u), u \rangle \\
&=\Big(\frac{1}{p}-\frac{1}{2q}\Big)\|u\|_{E}^p+\Big(\frac{1}{2q}-\frac{1}{p_{*}(\alpha)}\Big)\int_{\R^{N}} |x|^{-\alpha}|u|^{p_{*}(\alpha)} dx \\
&\geq \Big(\frac{1}{p}-\frac{1}{2q}\Big)\|u\|_{E}^p.
\end{aligned}
$$
By taking $c=\frac{1}{p}-\frac{1}{2q}$ we could reach to our conclusion.

\smallskip

\item [(ii)] Here, we use the fact that the embeddings  $E \hookrightarrow L^{s}(\R^N)$ and  $E \hookrightarrow L^{\frac{2Nq}{2N-\beta}}(\R^N)$ are continuous, for $p \leq s \leq p^{*}$ together with the Hardy-Littlewood-Sobolev inequality. For any $u\in {\mathcal N}$, we have
$$
\begin{aligned}
0=\langle {\mathcal J}'(u),u\rangle& =\|u\|_{E}^p-\int_{\R^{N}}|x|^{-\alpha}|u|^{p_{*}(\alpha)}-\lambda_{2} \int_{\R^{N}}\Big(|x|^{-\beta}*|u|^{q}\Big)|u|^{q}\\
&\geq \|u\|_{E}^p-S^{-\frac{p_{*}(\alpha)}{p}}_{p_{*}(\alpha)}\|u\|_{E}^{p_{*}(\alpha)}-C_{0}\lambda_{2} \|u\|_{E}^{2q}.
\end{aligned}
$$
Then, there exists some constant $C_1>0$ such that
\begin{equation}\label{cnot}
\|u\|_{E}\geq C_1>0\quad\mbox{for all }u\in {\mathcal N}.
\end{equation}
Next, we use the fact that ${\mathcal J} \!\mid_{\mathcal N}$ is coercive together with\eqref{cnot} in order to obtain
$$
{\mathcal J}(u) \geq \Big(\frac{1}{p}-\frac{1}{2q}\Big) C_1^p>0.
$$
\end{itemize}
\end{proof}
	
\begin{lemma}\label{nehari2}
Any critical point $u$ of ${\mathcal J} \!\mid_{\mathcal N}$ will be a free critical point.
\end{lemma}
\begin{proof}
Let $ {\mathcal K}(u)=\langle {\mathcal J}'(u),u\rangle $ for any $u \in E$. By using \eqref{cnot}, one could get
\begin{equation}\label{cnot1}
\begin{aligned}
\langle {\mathcal K}'(u),u\rangle&=p\|u\|^p-p_{*}(\alpha)\int_{\R^{N}}|x|^{-\alpha}|u|^{p_{*}(\alpha)}-2q \lambda_{2}\int_{\R^{N}}\Big(|x|^{-\beta}*|u|^{q}\Big)|u|^{q}\\
&=(p-2q)\|u\|_{E}^p+(2q-p_{*}(\alpha))\int_{\R^{N}}|x|^{-\alpha}|u|^{p_{*}(\alpha)}\\
&\leq -(2q-p)\|u\|_{E}^p\\
&<-(2q-p)C_1,
\end{aligned}
\end{equation}
for any $u \in {\mathcal N}$. Suppose that $u$ is a critical point of ${\mathcal J}$ in ${\mathcal N}$. By the Lagrange multiplier theorem, we obtain that there exists $\epsilon \in \R$ such that ${\mathcal J}'(u)=\epsilon {\mathcal K}'(u)$. Therefore, we have $\langle {\mathcal J} '(u),u\rangle=\epsilon \langle {\mathcal K}'(u),u\rangle$. As $\langle {\mathcal K}'(u),u\rangle<0$, so $\epsilon=0$ and further, we get ${\mathcal J}'(u)=0$.
	
\end{proof}

\begin{lemma}\label{nehari3}
Let us suppose the sequence $(u_n)$ is a $(PS)$ sequence for ${\mathcal J} \!\mid_{\mathcal N}$. Then, it is also a $(PS)$ sequence for ${\mathcal J}$.
\end{lemma}
\begin{proof}
Suppose that $(u_n)\subset {\mathcal N}$ is a $(PS)$ sequence for ${\mathcal J} \!\mid_{\mathcal N}$. As,
$$
{\mathcal J}(u_n)\geq \Big(\frac{1}{p}-\frac{1}{2q}\Big)\|u_n\|_{E}^p,
$$
that is, $(u_n)$ is bounded in $E$. Next, we need to prove that ${\mathcal J}'(u_n)\to 0$. We could notice that,
$$
{\mathcal J}'(u_n)- \epsilon_n {\mathcal K}'(u_n)= {\mathcal J}' \!\mid_{\mathcal N}(u_n)= o(1),
$$
for some $\epsilon_n \in \R$, which yields
$$
\epsilon_n \langle {\mathcal K}'(u_n),u_n \rangle= \langle {\mathcal J} '(u_n),u_n \rangle + o(1)= o(1).
$$
Using \eqref{cnot1}, we get $\epsilon_n \to 0$ and therefore, we obtain ${\mathcal J}'(u_n) \to 0$.
\end{proof}

\subsection{Compactness}\label{compc}
We define the energy functional ${\mathcal I}: E \to \R$ by
$$
{\mathcal I}(u)=\frac{1}{p}\|u\|_{E}^{p}-\frac{1}{p_{*}(\alpha)}\intr |x|^{-\alpha}|u|^{p_{*}(\alpha)},
$$
and the corresponding Nehari manifold for ${\mathcal I}$ by
$$
{\mathcal N}_{{\mathcal I}}=\{u\in E \setminus\{0\}: \langle {\mathcal I}'(u),u\rangle=0\}.
$$
Let
$$
m_{\mathcal I}=\inf_{u\in {\mathcal N}_{\mathcal I}}{\mathcal I}(u).
$$

And we have,
\begin{equation*}
\langle {\mathcal I}'(u),u\rangle= \|u\|_{E}^p-\int_{\R^{N}}|x|^{-\alpha}|u|^{p_{*}(\alpha)}.
\end{equation*}

\begin{lemma}\label{compact}
Let $(u_n)\subset{\mathcal N}_{\mathcal I}$ is a $(PS)$ sequence of ${\mathcal J} \!\mid_{{\mathcal N}}$, that is, $({\mathcal J}(u_n))$ is bounded and ${\mathcal J}'\!\mid_{{\mathcal N}}(u_n)\to 0$ strongly in $H_{0}^{-p}(\R^N)$. Then, there exists a solution $u\in E$ of \eqref{fr} such that, when we replace the sequence $(u_n)$ with the subsequence, one could have either of the following alternatives:
	
\smallskip
	
\noindent $(i)$ $u_n\to u$ strongly in $E$;
	
or
	
\smallskip
	
\noindent $(ii)$ $u_n\rightharpoonup u$ weakly in $E$. Further, there exists a positive integer $l\geq 1$ and $l$ nontrivial  weak solutions to \eqref{squ}, that is, $l$ functions $u_1,u_2,\dots, u_l\in E$ and $l$ sequences of points $(w_{n,1})$, $(w_{n,2})$, $\dots$, $(w_{n,l})\subset \R^N$ such that the following conditions hold:
\begin{enumerate}
\item[(a)] $|w_{n,j}|\to \infty$ and $|w_{n,j}-w_{n,i}|\to \infty$  if $i\neq j$, $n\to \infty$;
\item[(b)] $ u_n-\sum_{j=1}^lu_j(\cdot+w_{n,j})\to u$ in $E$;
\item[(c)] $ {\mathcal J}(u_n)\to {\mathcal J}(u)+\sum_{j=1}^l {\mathcal I}(u_j)$.
\end{enumerate}
\end{lemma}
\begin{proof}
Since, $(u_n)\in E$ is a bounded sequence. So, there exists $u\in E$ such that, up to a subsequence, we have
\begin{equation}\label{firstconv}
\left\{
\begin{aligned}
u_n& \rightharpoonup u \quad\mbox{ weakly in }E,\\
u_n &\rightharpoonup u\quad\mbox{ weakly in }L^s(\R^N),\; p\leq s\leq p_{*},\\
u_n & \to u\quad\mbox{ a.e. in }\R^N.
\end{aligned}
\right.
\end{equation}
Using Lemma \ref{anbl} and \eqref{firstconv}, we get   $${\mathcal J}'(u)=0,$$ which gives us that $u\in E$ is a solution of \eqref{fr}. In case if $u_n\to u$ strongly in $E$, then $(i)$ holds.

\medskip
	
Let us assume that $(u_n)\in E$ does not converge strongly to $u$ and define $f_{n,1}=u_n-u$. Then $(f_{n,1})$ converges weakly (not strongly) to zero in $E$ and using Brezis-Lieb Lemma (see \cite{B1983}), we have
\begin{equation}\label{bl2}
\|u_n\|_{E}^{p_{*}(\alpha)}=\|u\|_{E}^{p_{*}(\alpha)}+\|f_{n,1}\|_{E}^{p_{*}(\alpha)}+o(1).
\end{equation}
Therefore,
\begin{equation*}
\lim_{n\rightarrow \infty} \int_{\R^N} \Big(|x|^{-\alpha} |u_n|^{p_{*}(\alpha)} - |x|^{-\alpha} |f_{n,1}|^{p_{*}(\alpha)}\Big) dx =\intr |x|^{-\alpha} |u|^{p_{*}(\alpha)} dx,
\end{equation*}
which implies
\begin{equation}\label{bl3}
\int_{\R^N}|x|^{-\alpha} |u_n|^{p_{*}(\alpha)} dx = \intr |x|^{-\alpha} |f_{n,1}|^{p_{*}(\alpha)} dx + \intr |x|^{-\alpha} |u|^{p_{*}(\alpha)} dx + o(1),
\end{equation}
Using equation \eqref{bl2} and \eqref{bl3} we have
\begin{equation}\label{est6}
{\mathcal J}(u_n)= {\mathcal J}(u)+{\mathcal I}(f_{n,1})+o(1).
\end{equation}
 Also, for any $h\in E$, we have
\begin{equation}\label{est7}
\langle{\mathcal I}'(f_{n,1}), h\rangle=o(1).
\end{equation}
Next, using Lemma \ref{nlocbl} we obtain
$$
\begin{aligned}
0=\langle {\mathcal J}'(u_n), u_n \rangle&=\langle {\mathcal J}'(u),u\rangle+\langle {\mathcal I}'(f_{n,1}), f_{n,1} \rangle+o(1)\\
&=\langle{\mathcal I}'(f_{n,1}), f_{n,1}\rangle+o(1).
\end{aligned}
$$
Hence,
\begin{equation}\label{est8}
\langle {\mathcal I}'(f_{n,1}), f_{n,1}\rangle=o(1).
\end{equation}
Also, we have
$$
\Gamma:=\limsup_{n\to \infty}\Big(\sup_{v\in \R^N} \int_{B_1(v)}|f_{n,1}|^{p_{*}(\alpha)}\Big)> 0.
$$

Therefore, one could find $w_{n,1}\in \R^N$ such that
\begin{equation}\label{est9}
\int_{B_1(w_{n,1})}|f_{n,1}|^{p_{*}(\alpha)}>\frac{\Gamma}{2}.
\end{equation}
So, for any sequence $(f_{n,1}(\cdot+w_{n,1}))$, there exists $u_1\in E$ such that, up to a subsequence, we have 
$$
\begin{aligned}
f_{n,1}(\cdot+w_{n,1})&\rightharpoonup u_1\quad\mbox{ weakly in } E,\\
f_{n,1}(\cdot+w_{n,1})&\to u_1\quad\mbox{ strongly in } L_{loc}^{p_{*}(\alpha)}(\R^N),\\
f_{n,1}(\cdot+w_{n,1})&\to u_1\quad\mbox{ a.e. in } \R^N.
\end{aligned}
$$
Next, by passing to the limit in \eqref{est9} we get
$$
\int_{B_1(0)}|u_{1}|^{p_{*}(\alpha)}\geq \frac{\beta}{2},
$$
which gives us, $u_1\not\equiv 0$. Since $(f_{n,1}) \rightharpoonup 0$ weakly in $E$, we get that $(w_{n,1})$ is unbounded. Therefore, passing to a subsequence, we obtain $|w_{n,1}|\to \infty$. Next, by using \eqref{est8}, we have ${\mathcal I}'(u_1)=0$, which implies that $u_1$ is a nontrivial solution of \eqref{squ}.
Next, we define
$$
f_{n,2}(x)=f_{n,1}(x)-u_1(x-w_{n,1}).
$$
Similarly as before, we have
$$
\|f_{n,1}\|^{p_{*}(\alpha)}=\|u_1\|^{p_{*}(\alpha)}+\|f_{n,2}\|^{p_{*}(\alpha)}+o(1).
$$
And
$$
\int_{\R^N} |x|^{-\alpha} |f_{n,1}|^{p_{*}(\alpha)} =\int_{\R^N} |x|^{-\alpha} |u_1|^{p_{*}(\alpha)}+\intr |x|^{-\alpha} |f_{n,2}|^{p_{*}(\alpha)}+o(1). 
$$
Therefore,
$$
{\mathcal I}(f_{n,1})={\mathcal I}(u_1)+{\mathcal I}(f_{n,2})+o(1).
$$
Using \eqref{est6}, we have
$$
{\mathcal J} (u_n)= {\mathcal J} (u)+{\mathcal I}(u_1)+{\mathcal I}(f_{n,2})+o(1).
$$
By using the same approach as above, we have
$$
\langle {\mathcal I}'(f_{n,2}),h\rangle =o(1)\quad\mbox{ for any }h\in E
$$
and
$$
\langle {\mathcal I}'(f_{n,2}), f_{n,2}\rangle =o(1).
$$
Now, if $(f_{n,2}) \to 0$ strongly, then we could take $l=1$ in the Lemma \ref{compact} to conclude the proof.

\smallskip 

Suppose that $f_{n,2}\rightharpoonup 0$ weakly (not strongly) in $E$ and we iterate the process $l$ times and we could find a set of sequences $(w_{n,j})\subset \R^N$, $1\leq j\leq l$ with 
$$
|w_{n,j}|\to \infty\quad\mbox{  and }\quad |w_{n,i}-w_{n,j}|\to \infty\quad\mbox{  as }\; n\to \infty, i\neq j
$$
and $l$ nontrivial solutions  $u_1$, $u_2$, $\dots$, $u_l\in E$ of \eqref{squ} such that, by letting
$$
f_{n,j}(x):=f_{n,j-1}(x)-u_{j-1}(x-w_{n,j-1})\,, \quad 2\leq j\leq l,
$$ 
we get
$$
f_{n,j}(x+w_{n,j})\rightharpoonup u_j\quad\mbox{weakly in }\; E
$$
and
$$
{\mathcal J}(u_n)= {\mathcal J}(u)+\sum_{j=1}^l {\mathcal I}(u_j)+{\mathcal I}(y_{n,l})+o(1).
$$
As ${\mathcal J} (u_n)$ is bounded and ${\mathcal I}(u_j)\geq b_{\mathcal I}$, we get the desired result by iterating the process a finite number of times.
\end{proof}

\begin{lemma}\label{corr1} 
For any $e\in (0,m_{\mathcal I})$, any sequence $(u_n)$ which is a $(PS)_e$ sequence of ${\mathcal J} \! \mid_{{\mathcal N}} $, is also relatively compact.  
\end{lemma}
\begin{proof}
 We assume $(u_n)$ is a $(PS)_e$ sequence of ${\mathcal J}$ in ${{\mathcal N}}$. Using Lemma \ref{compact} we have ${\mathcal I}(u_j)\geq m_{\mathcal I}$. Therefore, upto a subsequence $u_n\to u$ strongly in $E$, which gives us that $u$ is a solution of \eqref{fr}. 
\end{proof}

\medskip

In order to finish the proof of Theorem \ref{gstate}, we need the following result.

\begin{lemma}\label{flg}
$$
m<m_{\mathcal I}.
$$
\end{lemma}
\begin{proof}
Let us denote the ground state solution of \eqref{squ} by $R\in E$ and such solution exists (see \cite{AN2016} and references therein). Suppose that the projection of $R$ on ${\mathcal N}$ is $tR$, that is, $t=t(R)>0$ is the unique real number such that $tR\in {\mathcal N}$. Since, $R\in {\mathcal N}_{\mathcal I}$ and $tR\in {\mathcal N}$, we obtain
\begin{equation}\label{g1}
||R||^p= \int_{\R^N} |x|^{-\alpha} |R|^{p_{*}(\alpha)}
\end{equation}
and
$$
t^p\|R\|^p=t^{p_{*}(\alpha)}\int_{\R^N} |x|^{-\alpha} |R|^{p_{*}(\alpha)}+ \lambda_{2} t^{2q}\int_{\R^N} \Big(|x|^{-\beta}*|R|^{q}\Big)|R|^{q}.
$$
We could notice that $t<1$ from the above two equalities. Therefore, we obtain
	
\begin{equation*}
\begin{aligned}
m \leq {\mathcal J}(tR)&=\frac{1}{p}t^{p}\|R\|^{p}-\frac{1}{p_{*}(\alpha)}t^{p_{*}(\alpha)}\int_{\R^N} |x|^{-\alpha} |R|^{p_{*}(\alpha)}- \frac{\lambda_2}{2q}t^{2q} \int_{\R^N} \Big(|x|^{-\beta}*|R|^{q}\Big)|R|^{q}\\
&= \Big(\frac{t^{p}}{p}-\frac{t^{p_{*}(\alpha)}}{p_{*}(\alpha)}\Big)\|R\|^{p}-\frac{1}{2q}\Big(t^p||R||^p-t^{p_{*}(\alpha)}\int_{\R^N} |x|^{-\alpha} |R|^{p_{*}(\alpha)}\Big)\\
&= t^{p} \Big(\frac{1}{p}-\frac{1}{2q}\Big)\|R\|^{p}+t^{p_{*}(\alpha)}\Big(\frac{1}{2q}-\frac{1}{p_{*}(\alpha)}\Big)\|R\|^{p}\\
&< \Big(\frac{1}{p}-\frac{1}{2q}\Big)\|R\|^{p}+\Big(\frac{1}{2q}-\frac{1}{p_{*}(\alpha)}\Big)\|R\|^{p}\\
&< \Big(\frac{1}{p}-\frac{1}{p_{*}(\alpha)}\Big)\|R\|^{p} ={\mathcal I}(R)= m_{\mathcal I}.
\end{aligned}
\end{equation*}
Hence, we reach our conclusion.
\end{proof}

Further, using Ekeland variational principle, for any $n\geq 1$ there exists $(u_n) \in {\mathcal N}$ such that
\begin{equation*}
\begin{aligned}
{\mathcal J}(u_n)&\leq m+\frac{1}{n} &&\quad\mbox{ for all } n\geq 1,\\
{\mathcal J}(u_n)&\leq {\mathcal J}(\tilde{u})+\frac{1}{n}\|\tilde{u}-u_n\| &&\quad\mbox{ for all } \tilde{u} \in {\mathcal N} \;\;,n\geq 1.
\end{aligned}
\end{equation*}
Next, we could easily deduce that $(u_n) \in {\mathcal N}$ is a $(PS)_{e}$ sequence for ${\mathcal J}$ on ${\mathcal N}$. Using Lemma \ref{flg} and Lemma \ref{corr1} we obtain that, up to a subsequence $u_n \to u$ strongly in $E$ which is a ground state solution of the ${\mathcal J}$. 

\medskip

\section{Proof of Theorem 2}

In this section, we are concerned with the existence of a least energy sign-changing solution of \eqref{fr}. 

\subsection{Proof of Theorem }

\begin{lemma}\label{frl1}
Let us assume that $N\geq 5$, $p_{*}(\alpha)>2q>p$ and $\lambda_{2} \in \R$. There exists a unique pair $(\bar{\theta_1}, \bar{\theta_2})\in (0, \infty)\times (0, \infty)$ such that, for any $u \in  E$ and $u^{\pm} \neq 0$, $\bar{\theta_1} u^{+}+\bar{\theta_2} u^{-} \in {\cal \overline{N}}$. Also, for any $u\in {\cal \overline{N}}$, ${\mathcal J}(u)\geq {\mathcal J}(\theta_1 u^{+}+\theta_2 u^{-})$ for all $\theta_1, \theta_2 \geq 0$.
\end{lemma}

\begin{proof}
We will be following the idea of \cite{VX2017} in order to prove this result. Let us define the function $ \varphi: [0, \infty)\times [0, \infty)\rightarrow \R$ by
\begin{equation*}
\begin{aligned}
\varphi(\theta_1, \theta_2)&= {\mathcal J}(\theta_1^{\frac{1}{2p_{*}(\alpha)}} u^{+}+\theta_2^{\frac{1}{2p_{*}(\alpha)}} u^{-})\\
&= \frac{\theta_1^{\frac{p}{2p_{*}(\alpha)}}}{p}\|u^{+}\|_{E}^{p}+\frac{\theta_2^{\frac{p}{2p_{*}(\alpha)}}}{p}\|u^{-}\|_{E}^{p}-\lambda_2\frac{\theta_1^{\frac{q}{p_{*}(\alpha)}}}{2q}\int_{\R^{N}}\Big(|x|^{-\beta}*(u^{+})^{q}\Big)(u^{+})^{q}\\&-\lambda_2\frac{\theta_2^{\frac{q}{p_{*}(\alpha)}}}{2q}\int_{\R^{N}}\Big(|x|^{-\beta}*(u^{-})^{q}\Big)(u^{-})^{q}-\lambda_{2}\frac{\theta_1^{\frac{q}{2p_{*}(\alpha)}}\theta_2^{\frac{q}{2p_{*}(\alpha)}}}{2q}\int_{\R^{N}}\Big(|x|^{-\beta}*(u^{+})^{q}(u^{-})^{q}\\&-\frac{\theta_{1}^{\frac{1}{2}}}{p_{*}(\alpha)}\int_{\R^{N}}|x|^{-\alpha}(u^{+})^{p_{*}(\alpha)} -\frac{\theta_{2}^{\frac{1}{2}}}{p_{*}(\alpha)}\int_{\R^{N}}|x|^{-\alpha}(u^{-})^{p_{*}(\alpha)}.
\end{aligned}
\end{equation*}

Note that $\varphi$ is strictly concave and therefore, $\varphi$ has at most one maximum point. We also have
\begin{equation}\label{fr2}
\lim_{\theta_1 \rightarrow \infty}\varphi(\theta_1, \theta_2)= -\infty \mbox{ for all }\theta_2 \geq 0 \quad\mbox{ and } \quad\mbox{ } \lim_{\theta_2 \rightarrow \infty}\varphi(\theta_1, \theta_2)= -\infty \mbox{ for all }\theta_1 \geq 0,
\end{equation}
and one could notice that
\begin{equation}\label{fr3}
\lim_{\theta_1 \searrow 0}\frac{\partial{\varphi}}{\partial{\theta_1}}(\theta_1, \theta_2)= \infty \mbox{ for all }\theta_2> 0 \quad\mbox{ and }  \lim_{\theta_2 \searrow 0}\frac{\partial{\varphi}}{\partial{\theta_2}}(\theta_1, \theta_2)= \infty \mbox{ for all }\theta_1> 0.
\end{equation}
Using \eqref{fr2} and \eqref{fr3}, it could be seen that maximum cannot be achieved at the boundary. Hence, $\varphi$ has exactly one maximum point $(\bar{\theta_1}, \bar{\theta_2})\in (0, \infty)\times (0, \infty)$.

\end{proof}

\medskip

We divide our proof into two steps.

\medskip

\noindent \text{Step 1.}{\it \;\; The energy level $\overline{m}>0$ is achieved by some $\sigma \in {\cal \overline{N}}$.}

\medskip

Let $(u_n)\subset {\cal \overline{N}}$ be a minimizing sequence for $\overline{m}$. We have 
\begin{equation*}
\begin{aligned}
{\mathcal J}(u_{n})&= {\mathcal J}(u_{n})-\frac{1}{2q}\langle {\mathcal J}'(u_{n}), u_{n}\rangle \\
&=\Big(\frac{1}{p}-\frac{1}{2q}\Big)\|u_n\|_{E}^p+\Big(\frac{1}{2q}-\frac{1}{p_{*}(\alpha)}\Big)\int_{\R^{N}} |x|^{-\alpha}|u_n|^{p_{*}(\alpha)}  \\
&\geq \Big(\frac{1}{p}-\frac{1}{2q}\Big)\|u_n\|_{E}^p\\
&\geq C_1\|u_{n}\|_{E}^{p},
\end{aligned}
\end{equation*}
for some positive constant $C_1>0$. Therefore, for $C_{2}> 0$ one has
$$
\|u_{n}\|_{E}^{p}\leq C_{2}{\mathcal J}(u_{n})\leq M,
$$
that is, $(u_n)$ is bounded in $E$. Hence, $(u_{n}^{+})$ and $(u_{n}^{-})$ are also bounded in $E$. By passing to a subsequence, there exists $u^{+}$, $u^{-}\in E$ such that
$$
u_{n}^{+}\rightharpoonup u^{+} \mbox{ and } u_{n}^{-}\rightharpoonup u^{-} \quad\mbox{ weakly in } E.
$$
Since $q$ satisfy \eqref{p1}, we have that the embeddings $E\hookrightarrow L^{s}(\R^{N})$ and  $E\hookrightarrow L^{\frac{2Nq}{2N-\beta}}(\R^{N})$ are compact for $p< s < p_{*}$. Hence,
\begin{equation}\label{m1}
u_{n}^{\pm} \rightarrow u^{\pm} \quad\mbox{ strongly in } L^{p_{*}(\alpha)}(\R^{N}) \cap L^{\frac{2Nq}{2N-\beta}}(\R^{N}).
\end{equation}
Next, by using the Hardy-Littlewood-Sobolev inequality, we get
\begin{equation*}
\begin{aligned}
C\Big(\|u_{n}^{\pm}\|_{L^{p_{*}(\alpha)}(\R^N)}^{p}+\|u_{n}^{\pm}\|_{L^{\frac{2Nq}{2N-\beta}}}^{p}\Big)&\leq \|u_{n}^{\pm}\|_{E}^{p}\\
&= \int_{\R^{N}}|x|^{-\alpha}|u_{n}^{\pm}|^{p_{*}(\alpha)}+|\lambda_2|\int_{\R^{N}}\Big(|x|^{-\beta}*|u_{n}^{\pm}|^{q}\Big)|u_{n}^{\pm}|^{q}\\
&\leq C\Big(\|u_{n}^{\pm}\|_{L^{p_{*}(\alpha)}(\R^N)}^{p_{*}(\alpha)}+\|u_{n}^{\pm}\|_{L^{\frac{2Nq}{2N-\beta}}}^{q}\Big)\\
&\leq C\Big(\|u_{n}^{\pm}\|_{L^{p_{*}(\alpha)}(\R^N)}^{p}+\|u_{n}^{\pm}\|_{L^{\frac{2Nq}{2N-\beta}}}^{p}\Big) \Big(\|u_{n}^{\pm}\|_{L^{p_{*}(\alpha)}(\R^N)}^{p_{*}(\alpha)-p}+\|u_{n}^{\pm}\|_{L^{\frac{2Nq}{2N-\beta}}}^{q-p}\Big).
\end{aligned}
\end{equation*}
Since $u_{n}^{\pm}\neq 0$, we get
\begin{equation}\label{m2}
\Big(\|u_{n}^{\pm}\|_{L^{p_{*}(\alpha)}(\R^N)}^{p_{*}(\alpha)-p}+\|u_{n}^{\pm}\|_{L^{\frac{2Nq}{2N-\beta}}}^{q-p}\Big)\geq C> 0 \quad\mbox{ for all } n\geq 1.
\end{equation}
Using \eqref{m1} and \eqref{m2}, we have that $u^{\pm} \neq 0$. Next, we use \eqref{m1} together with Hardy-Littlewood-Sobolev inequality and deduce
\begin{equation*}
\begin{aligned}
&\int_{\R^{N}}|x|^{-\alpha} (u_{n}^{\pm})^{p_{*}(\alpha)} &&\int_{\R^{N}}|x|^{-\alpha} (u^{\pm})^{p_{*}(\alpha)},\\
&\int_{\R^{N}}\Big(|x|^{-\beta}*(u_n^{\pm})^{q}\Big)(u_n^{\pm})^{q} &&\rightarrow \int_{\R^{N}}\Big(|x|^{-\beta}*(u^{\pm})^{q}\Big)(u^{\pm})^{q},
\end{aligned}
\end{equation*}
and
\begin{equation*}
\int_{\R^{N}}\Big(|x|^{-\beta}*(u_n^{+})^{q}\Big)(u_n^{-})^{q} \rightarrow \int_{\R^{N}}\Big(|x|^{-\beta}*(u^{+})^{q}\Big)(u^{-})^{q}.
\end{equation*}
Also, by Lemma \ref{frl1}, there exists a unique pair $(\bar{\theta_1}, \bar{\theta_2})$ such that $\bar{\theta_1}u^{+}+\bar{\theta_2} u^{-}\in {\cal \overline{N}}$. Since the norm $\|.\|_{E}$ is weakly lower semi-continuous, we deduce that
\begin{equation*}
\begin{aligned}
\overline{m} \leq {\mathcal J}(\bar{\theta_1}u^{+}+\bar{\theta_2} u^{-})&\leq \liminf_{n\rightarrow \infty} {\mathcal J}(\bar{\theta_1}u^{+}+\bar{\theta_2} u^{-})\\
&\leq \limsup_{n\rightarrow \infty} {\mathcal J}(\bar{\theta_1}u^{+}+\bar{\theta_2} u^{-})\\
&\leq \lim_{n\rightarrow \infty}{\mathcal J}(u_{n})\\
&= \overline{m}.
\end{aligned}
\end{equation*}
Finally, we take $\sigma= \bar{\theta_1}u^{+}+\bar{\theta_2} u^{-} u^{-}\in {\cal \overline{N}}$ in order to conclude the proof.

\medskip

\noindent \text{Step 2.}{ \it \;\;  $\sigma \in {\cal \overline{N}}$ is the critical point of ${\mathcal J}:E \rightarrow \R$.
}

\medskip

Let us assume that the $\sigma$ is not a critical point of ${\mathcal J}$. Then there exists $\tau\in C_{c}^{\infty}(\R^N)$ such that $ \langle {\mathcal J}'(\sigma), \tau \rangle= -2.$ Since, ${\mathcal J}$ is continuous and differentiable, there exists $\Xi>0$ small such that
\begin{equation}\label{fr4}
\langle {\mathcal J}'(\theta_1 u^{+}+\theta_2 u^{-}+\omega \bar{\sigma}), \bar{\sigma} \rangle \;\; \leq -1 \quad\mbox{ if } (\theta_1- \bar{\theta_1})^{2}+(\theta_2- \bar{\theta_2})^{2}\leq \Xi^{2} \mbox{ and } 0\leq \omega \leq \Xi.
\end{equation}
Let $D\subset \R^{2} $ be an open disc of radius $\Xi>0$ centered at $(\bar{\theta_1}, \bar{\theta_2})$. Also, define a continuous function $\varPhi: D\rightarrow [0, 1]$ by 
\begin{equation*}
\varPhi(\theta_1, \theta_2)= \left\{\begin{array}{cc}1\quad\mbox{ if }(\theta_1- \bar{\theta_1})^{2}+(\theta_2- \bar{\theta_2})^{2}\leq \frac{\Xi^{2}}{16}, \\ 0\quad\mbox{ if }(\theta_1- \bar{\theta_1})^{2}+(\theta_2- \bar{\theta_2})^{2}\geq \frac{\Xi^{2}}{4}.\end{array} \right.
\end{equation*}
Next, we define a continuous map $T: D\rightarrow E$ by 
\begin{equation*}
T(\theta_1, \theta_2)= \theta_1 u^{+}+\theta_2 u^{-}+\Xi \varPhi(\theta_1, \theta_2)\bar{\sigma} \quad\mbox{ for all } (\theta_1, \theta_2)\in D
\end{equation*}
and $Q: D\rightarrow \R^{2}$ by
\begin{equation*}
Q(\theta_1, \theta_2)= (\langle {\mathcal J}'(T(\theta_1, \theta_2)), T(\theta_1, \theta_2)^{+}\rangle, \langle {\mathcal J}'(T(\theta_1, \theta_2)), T(\theta_1, \theta_2)^{-}\rangle) \quad\mbox{ for all }(\theta_1, \theta_2)\in D. 
\end{equation*}
One could deduce that $Q$ is continuous as the mapping $u \mapsto u^{+}$ is continuous in $E$. In case, we are on the boundary of $D$, that is,   $(\theta_1- \bar{\theta_1})^{2}+(\theta_2- \bar{\theta_2})^{2}= \Xi^{2}$, then $\varPhi= 0$. Hence, one could get  $T(\theta_1, \theta_2)= \theta_1 u^{+}+\theta_2 u^{-}$ and by Lemma \ref{frl1}, we have
\begin{equation*}
Q(\theta_1, \theta_2)\neq 0 \quad\mbox{ on } \partial{D}.
\end{equation*}
Therefore, the Brouwer degree is well defined and $\deg(Q, {\rm int} (D), (0, 0))=1$ and there exists $(\theta_{11}, \theta_{21})\in {\rm int} (D)$ such that $Q(\theta_{11}, \theta_{21})= (0, 0)$. This further implies that $T(\theta_{11}, \theta_{21})\in {\cal \overline{N}}$ and by the definition of $\overline{m}$ one could deduce
\begin{equation}\label{fr5}
{\mathcal J}(T(\theta_{11}, \theta_{21}))\geq \overline{m}.
\end{equation}
Using the equation \eqref{fr4}, we get
\begin{equation}\label{fr6}
\begin{aligned}
{\mathcal J}(T(\theta_{11}, \theta_{21}))&= {\mathcal J}(\theta_{11} u^{+}+\theta_{21} u^{-})+\int_{0}^{1}\frac{d}{dt}{\mathcal J}(\theta_{11} u^{+}+\theta_{21} u^{-}+\Xi t \varPhi(\theta_{11}, \theta_{21})\bar{\sigma})dt \\
&= {\mathcal J}(\theta_{11} u^{+}+\theta_{21} u^{-})-\Xi \varPhi(\theta_{11}, \theta_{21}).
\end{aligned}
\end{equation}
Next, by definition of $\varPhi$ we get $\varPhi(\theta_{11}, \theta_1)=1$ when $(\theta_{11}, \theta_{21})=(\bar{\theta_1}, \bar{\theta_2})$. Therefore, we have
$$
{\mathcal J}(T(\theta_{11}, \theta_{21}))\leq {\mathcal J}(\theta_{11} u^{+}+\theta_{21} u^{-})-\Xi\leq \overline{m}-\Xi< \overline{m}.
$$
When $(\theta_{11}, \theta_{21})\neq (\bar{\theta_1}, \bar{\theta_2})$, then by Lemma \ref{frl1} we have 
$$
{\mathcal J}(\theta_{11} u^{+}+\theta_{21} u^{-})< {\mathcal J}(\bar{\theta_1} u^{+}+\bar{\theta_2} u^{-})= \overline{m},
$$
which yields
$$
{\mathcal J}(T(\theta_{11}, \theta_{21}))\leq {\mathcal J}(\theta_{11} u^{+}+\theta_{21} u^{-})< \overline{m},
$$
a contradiction to the equation \eqref{fr5}. Hence, we conclude our proof.

\end{document}